\let\olddiv\div

\documentclass[a4paper, 11pt, final,underruledhead]{paper}
\usepackage[mathscr]{euscript}
\usepackage{enumerate, amssymb, math, mathtools}
\usepackage{graphics}
\usepackage[dvips]{graphicx}
\usepackage{caption, subcaption, float}

\begin{document}



\let\goth\mathfrak


\def\myend{{}\hfill{\small$\bigcirc$}}

\newenvironment{ctext}{%
  \par
  \smallskip
  \centering
}{%
 \par
 \smallskip
 \csname @endpetrue\endcsname
}


\def\id{\mathrm{id}}

\newcommand{\msub}{\mbox{\large$\goth y$}}    

\def\LineOn(#1,#2){\overline{{#1},{#2}\rule{0em}{1,5ex}}}
\def\lines{{\cal L}}
\def\projgeo(#1){PG({#1})}

\def\VerSpace(#1,#2){{\bf V}_{{#2}}({#1})}

\def\VerSpacex(#1,#2){{\bf V}^\ast_{{#2}}({#1})}
\def\GrasSpace(#1,#2){{\bf G}_{{#2}}({#1})}
\def\vergras{{\goth R}}
\def\konftyp(#1,#2,#3,#4){\left( {#1}_{#2}\, {#3}_{#4} \right)}
\def\binokonf(#1){\konftyp({\binom{{#1}}{2}},{{#1}-2},\binom{{#1}}{3},3)}
\let\binoconf\binokonf

\newcount\liczbaa
\newcount\liczbab

\def\binkonfo(#1,#2){\liczbaa=#2 \liczbab=#2 \advance\liczbab by -2
\def\doa{\ifnum\liczbaa = 0\relax \else
\ifnum\liczbaa < 0 \the\liczbaa \else +\the\liczbaa\fi\fi}
\def\dob{\ifnum\liczbab = 0\relax \else
\ifnum\liczbab < 0 \the\liczbab \else +\the\liczbab\fi\fi}
\konftyp(\binom{#1\doa}{2},#1\dob,\binom{#1\doa}{3},3) }

\newcount\liczbac
\def\binkonf(#1,#2){\liczbac=#2 
\def\docc{\ifnum\liczbac = 0 \relax \else\ifnum\liczbac < 0\the\liczbac 
\else+\the\liczbac\fi\fi}
B_{#1{\docc}}}


\def\PSTS{{\sf PSTS}}
\def\BSTS{{\sf BSTS}}


\title[Hyperplanes of binomial \PSTS's]{%
Binomial partial Steiner triple systems with complete graphs:
structural problems}

\author{Krzysztof Petelczyc,  %
Ma{\l}gorzata Pra{\.z}mowska, %
Krzysztof Pra{\.z}mowski}



\maketitle

\section*{Introduction}

In the paper we study the structure of hyperplanes of so called binomial partial Steiner triple
systems (\BSTS's, in short)
i.e. of configurations with $\binom{n}{2}$ points and $\binom{n}{3}$ lines, 
each line of the size $3$. Consequently, a \BSTS\ has $n-2$ lines through each of its points.

The notion of a hyperplane is commonly used within widely understood geometry.
Roughly speaking, a hyperplane of a (geometrical) space $\goth M$ is a maximal proper subspace of
$\goth M$. A more specialized characterization of a (``geometrical'') hyperplane comes from 
projective geometry: a hyperplane of a (partial linear = semilinear) space $\goth M$
is a proper subspace of $\goth M$ which crosses every line of $\goth M$.
Note that these two characterizations are not equivalent in general.
In the context of incidence geometry the second characterization is primarily used
(cf. \cite{veldkamp} or \cite{ronan}),
and also in our paper in investigations on some classes of partial Steiner triple systems 
(in short: \PSTS's)  we shall follow this approach.
For a \PSTS\ $\goth M$ there is a natural structure of a projective space with all the lines 
of size $3$ definable on the family of all hyperplanes of $\goth M$
(the so called Veldkamp space of $\goth M$).
On other side our previous investigations on \PSTS's and graphs contained in them lead us to characterizations
of systems which freely contain complete graphs (one can say, informally and not really exactly:
systems freely generated by a complete graph);
these all fall into the class of so called binomial configurations i.e. 
$\left( {\binom{\nu+\varkappa-1}{\nu}}_{\nu}\; {\binom{\nu+\varkappa-1}{\varkappa}}_{\varkappa} \right)$-configurations 
with integers $\nu,\varkappa\geq 2$.
A characterization of \PSTS's which freely contain at least given number $m$ of complete subgraphs
appeared available, and for particular values of $m$ a complete classification of the resulting
configurations was proved (see \cite{STP-4}).
It turned out so, that the structure of complete subgraphs of $\goth M$ says much about $\goth M$,
but fairly it does not determine $\goth M$.

Now, quite surprisingly, we have observed that the complement of such a free complete subgraph of
a \PSTS\ $\goth M$ is a hyperplane of $\goth M$.
So, our previous classification is equivalent to 
characterizations and classifications of binomial \PSTS's based on the structure of
their binomial hyperplanes.
But a \PSTS, if contains a binomial hyperplane, usually contains also other (non-binomial)
hyperplanes. So, the structure of all the hyperplanes of a \PSTS\ $\goth M$ says much more 
about the structure of $\goth M$.
In the paper we have determined the structure of hyperplanes of \PSTS's of some important
classes, in particular of so called generalized Desargues configurations 
(cf. \cite{doliwa1}, \cite{doliwa2}, \cite{perspect}, \cite{saniga}),
of combinatorial Veronese structures and of dual combinatorial Veronese structures, both with 
3-element lines (cf. \cite{combver}, \cite{veradjac}), and of so called combinatorial 
quasi Grassmannians (cf. \cite{skewgras}). Exact definitions of respective classes of
configurations are quoted in the text. 
Beautiful figures illustrating the schemes of hyperplanes in small \PSTS's were prepared by Krzysztof Petelczyc.
We have also shown a general method to characterize
all the hyperplanes in an arbitrary \BSTS\ with at least one ((maximal) free complete subgraph
(Theorems \ref{prop:hipinSSP:0}, \ref{prop:hipinSSP:2}).

As it was said: the hyperplanes of a \PSTS\ yield a projective space $\goth P$.
In essence, ${\goth P} = \projgeo(n,2)$ for some integer $n$, so only 
$n = \dim({\goth P})$ is an important parameter, but non-isomorphic \PSTS's may have
the same number $2^{n+1}-1$ of hyperplanes. Consequently, the projective space of hyperplanes
of a binomial \PSTS\ $\goth M$ does not give a complete information on the geometry of $\goth M$.

However, if the points of the $\projgeo(n,2)$, associated  with a \BSTS, are labelled by the
type of geometry that respective hyperplanes carry, the number 
of nonisomorphic realizations of such labelled spaces drastically decreases.
It is pretty well seen in the case of $10_3$-configurations, but one can 
observe it for all \BSTS\ with arbitrary rank of points.

\section{Binomial subspaces of a \BSTS}

A {\em partial Steiner triple system} (a \PSTS) is a partial linear space 
${\goth M} = \struct{S,\lines}$ with the constant point rank
and all the lines of the size $3$.
A {\em binomial partial Steiner triple system} (a \BSTS) is a configuration
of the type $\binkonfo(n,0)$ for an integer $n\geq 4$; for short, we write
$\binkonf(n,0)$ for a configuration with these parameters.
\par
The symbols $\sub(X)$ and $\sub_k(X)$ stand for the subsets and the $k$-subsets
of a set $X$, resp.

\subsection{The structure of maximal free subgraphs}\label{ss:gen}

A \PSTS\ $\goth M$ {\em freely contains the complete graph} $K_X$, $X \subset S$
iff for any disjoint $2$-subsets $\{ a_1,a_2 \}$ and $\{ b_1,b_2 \}$ of $X$ we have
  $\LineOn(a_1,a_2) \cap \LineOn(b_1,b_2) = \emptyset$
($\overline{a}$ denotes the line of $\goth M$ which contains $a$)
and no $3$-subset of $X$ is on a line of $\goth M$.

Let us recall after \cite{klik:binom}
some basic properties of \BSTS's.
\begin{prop}\label{minSTS2pelny}
  Let $n\geq 2$ be an integer.
  A smallest \PSTS\ that freely contains the complete graph $K_n$ is a 
  $\binkonf(n,+1)$-configuration.
  Consequently, it is a \BSTS.
\end{prop}
\begin{prop}\label{pelny2horyzont}
  Let ${\goth M} = \struct{S,\lines}$ be a minimal \PSTS\ which freely contains a complete
  graph $K_X = \struct{X,{\sub_2(X)}}$ and $|X| = n$. 
  Then {\em the complement of $K_X$}, i.e. the structure
  \begin{equation}
    {\goth M}\setminus X :=
    \struct{S\setminus X,\lines\setminus\{ \overline{e}\colon e\in{\sub_2(X)} \}}
  \end{equation}
  is a
  $\binkonf(n,0)$-configuration and a subspace of $\goth M$.
  \par
  Conversely, let $\goth M$ contain as a subspace a
    $\binkonf(n,0)$-configuration ${\goth N} = \struct{Z,{\cal G}}$,
  Then $S\setminus Z$ yields in $\goth M$ a complete $K_n$-graph
  freely contained in $\goth M$, 
  whose complement is $\goth N$.
\end{prop}
\begin{prop}\label{prop:cross-compl2}
  Any two distinct complete $K_n$-graphs freely contained in a \linebreak 
  $\binkonf(n,+1)$-configuration
  share exactly one vertex.
\end{prop}
\begin{prop}\label{prop:cross-compl-line}
  Let $\struct{X_i,\sub_2(X_i)}$, $i=1,2,3$ be three distinct $K_n$ graphs 
  freely contained in a $\binkonf(n,+1)$-configuration $\goth M$. Let 
  $c_k\in X_i\cap X_j$ for all $\{k,i,j\} = \{1,2,3\}$.
  Then $\{c_1,c_2,c_3\}$ is a line of $\goth M$.
\end{prop}
%

\subsection{Algebra of hyperplanes}

\def\hipy{{\mathscr H}}
\def\hipcap{\pitchfork}
\def\hipa{{\cal H}}
\def\syminus{\olddiv}
\def\VSpace(#1){{\mbox{\boldmath$V$}}(#1)}

Let ${\cal Z}_1,{\cal Z}_2$ be two subsets of a set $S$.
We write (cf. \cite{semipap})
\begin{eqnarray}
  {\cal Z}_1\hipcap{\cal Z}_2 & := &
	 ({\cal Z}_1\cap{\cal Z}_2) \cup 
	 \big((S\setminus {\cal Z}_1)\cap(S\setminus{\cal Z}_2 )\big)
  \\
  \strut & \;= & 
    S\setminus ({\cal Z}_1  \syminus {\cal Z}_2),
\end{eqnarray}
where $\syminus$ denotes the operation of symmetric difference.
Note that identifying a subset $\cal Y$ of $S$ 
with its characteristic function $\chi_{\cal Y}$,
and, consequently,  identifying $S$ with the constant function $\bf 1$ we can compute simply 
  $S\setminus {\cal Y} = {\bf 1}+{\cal Y}$. 
After that we have 
  ${\cal Y}_1\hipcap{\cal Y}_2 = {\bf 1}+({\cal Y}_1+{\cal Y}_2)$.
Simple computations in the $Z_2$-algebra of characteristic functions of subsets of $S$
yield immediately the following equations valid for arbitrary subsets 
${\cal Y},{\cal Y}_1,{\cal Y}_2$ of $S$:
\begin{eqnarray}
   {\cal Y}\hipcap {\cal Y} & = & S,
   \\
   {\cal Y}\hipcap S & = & {\cal Y},
   \\
   {\cal Y}_1\hipcap {\cal Y}_2 & = & {\cal Y}_2\hipcap {\cal Y}_1,
   \\
   {\cal Y}_1\hipcap({\cal Y}_1\hipcap {\cal Y}_2) & = & {\cal Y}_2,
   \\
   ({\cal Y}_1\hipcap {\cal Y}_2) \cap {\cal Y}_2 & = & {\cal Y}_1 \cap {\cal Y}_2,
   \\
   ({\cal Y}\hipcap {\cal Y}_1)\hipcap({\cal Y}\hipcap {\cal Y}_2) & = & {\cal Y}_1\hipcap {\cal Y}_2,
\end{eqnarray}
Formally, the operation $\hipcap$ depends on the superset $S$ which contains the arguments of $\hipcap$.
In what follows we shall frequently use this operation without fixing $S$ explicitly:
the role of $S$ will be seen from the context.
\par
A {\em hyperplane of a \PSTS}\ $\goth M$ is an arbitrary proper subspace of $\goth M$ which crosses
every line of $\goth M$.
\begin{prop}\label{prop:veldkam}
  If $H_1,H_2$ are distinct hyperplanes of $\goth M$ then 
  $H_1\hipcap H_2$ is a hyperplane of $\goth M$ as well.
\end{prop}
\begin{proof}
  Let $L$ be a line of $\goth M$. Set $H = H_1\hipcap H_2$.
  Write $L = \{ q_1,q_2,q_3 \}$. It is seen that, up to a numbering of variables, one of the following
  must occur:
  \begin{sentences}\itemsep-2pt
  \item
    $q_1 \in H_1,H_2$, $q_2,q_3\notin H_1\cup H_2$: then $L\subset H$.
  \item
    $q_1\in H_1,H_2$, $q_2,q_3\in H_1\setminus H_2$.
  \item
    $q_1,q_2,q_3\in H_1,H_2$: clearly, $L\subset H$.
  \item
    $q_1\in H_1\setminus H_2$, $q_2\in H_2\setminus H_1$, $q_3\notin H_1,H_2$:
    then $q_3\in H$.
  \end{sentences}
  In each case $L$ crosses $H$, and if $L$ has two points in $H$ then
  $L$ is contained in $H$.
\end{proof}
Let $\hipy({\goth M})$ be the set of hyperplanes of $\goth M$.
Note, in addition, that the structure
\begin{equation}\label{def:VS}
  {\VSpace({\goth M})} := \struct{\hipy(\goth M), 
    \big\{ \{ H_1,H_2,H_1\hipcap H_2 \}\colon
    H_1,H_2\in\hipy({\goth M}),\, H_1\neq H_2\big\} }
\end{equation}
is a projective space $\projgeo(n,2)$, possibly degenerated i.e. with $n=-1,0,1$ allowed.
This projective space will be referred to as 
{\em the Veldkamp space of} $\goth M$
(cf. \cite{veldkamp}, \cite{shult}, (or \cite{ronan})).

As a by-product we get that 
{\em for each \PSTS\ $\goth M$, $|\hipy({\goth M})| = 2^{n+1} - 1$ for an integer  $n \geq -1$}.

For an arbitrary set $X$ and $\emptyset\neq A',A''\subset X$ we write 
\begin{equation}\label{def:dziel2hip}
   \hipa(A'|A'') := \sub_2(A')\cup \sub_2(A'').
\end{equation}

The following set-theoretical formula
\begin{multline}\label{wzor:hipcap}
 \hipa(A|X\setminus A)\hipcap \hipa(B|X\setminus B) =
 \\
 \hipa\Big( (A\cap B) \cup \big((X\setminus A)\cap(X\setminus B) \big) \Big{|}
 \big(A\cap (X\setminus B)\big)\cup\big( B\cap (X\setminus A) \big)  \Big)
 \\
 = \hipa(A \syminus B| X\setminus (A\syminus B))
\end{multline}
is valid for any distinct $\emptyset\neq A,B\subsetneq X$.

Let us fix $Z \subsetneq X$, $X$ -- finite. The following is just a simple though important
observation.
\begin{rem}\label{rem:algdziel}
  The set 
  $$
    \big\{\hipa(\{i\}|X\setminus\{i\})\colon i \in Z\big\}
  $$
  generates via $\hipcap$ the subalgebra 
  $$
    D(Z) = \big\{\hipa(A|X\setminus A)\colon \emptyset\neq A \subset Z\big\}
  $$
  of the $\hipcap$-algebra
  $$
    D(X) = \big\{\hipa(A|X\setminus A)\colon \emptyset\neq A \subsetneq X\big\}.
  $$
  Clearly, $D(Z)$ determines as in \eqref{def:VS}
  a Fano projective space $\projgeo(n,2)$, a subspace of the projective space in
  the analogous way associated with $D(X)$.
  Both algebras and both projective spaces up to an isomorphism depend
  entirely on the cardinalities $|Z|$ and $|X|$.
\end{rem}

\subsection{Binomial subconfigurations}

Next, we continue investigations of Subsection \ref{ss:gen},  
but now we concentrate upon
the `complementary configurations' contained in a \BSTS; in view of \ref{pelny2horyzont},
this is an equivalent approach.

Let us begin with a few words on basic properties of such a complementary configuration.
\begin{prop}\label{prop:hipowosc}
  Let $X$ be a $K_n$-graph freely contained in a 
    $\binkonf(n,1)$-configuration ${\goth M} = \struct{S,\lines}$ 
  and let $Y = S \setminus X$ be the corresponding
  $\binkonf(n,0)$-subconfiguration, complementary to $K_X$.
  Then $Y$ is a hyperplane of $\goth M$.
\end{prop}
\begin{proof}
  Let $L\in\lines$, if $x\in L\cap Y$ or $x,y\in L\cap Y$ are given, then $L\cap Y\neq\emptyset$:
  by assumptions. Suppose there are two $x,y\in L$, $x,y\notin Y$.
  So, $x,y\in X$, by assumptions on $X$ we have $\LineOn(x,y)\setminus X \in L\cap Y$.
\end{proof}

Roughly speaking, establishing the structure which maximal complete graphs
yield in a $\binkonf(n,1)$-\BSTS\ consists in establishing the structure 
which maximal binomial subspaces yield in the configuration, 
which can be equivalently reformulated
as establishing the structure of binomial hyperplanes in the binomial configuration.
So, the subject of this paper is the problem known as {\em hyperplanes arrangements}
in binomial partial Steiner triple systems.
  The question if each hyperplane is a complete-graph-complement has, generally a negative solution:
  indeed (cf. \ref{thm:hipingras}), if 
  $|A|,|X\setminus A|> 2$ then 
  $\hipa(A|X\setminus A)$ defined by \eqref{def:dziel2hip}
  is not a binomial configuration, though it happens to be even a hyperplane.
A first counterexample is given in \ref{thm:hipingras}.
A more general argument follows by \ref{prop:veldkam} and the following observation.
\begin{prop}
  Let $Y_1,Y_2$ be complements of distinct maximal complete graphs $X_1,X_2$
  freely contained in a \BSTS\ ${\goth M} = \struct{S,\lines}$, $|S| > 3$.
  Then $Y_1\hipcap Y_2$ is a hyperplane of $\goth M$ which is not  the complement of any 
  complete subgraph of $\goth M$.
\end{prop}
\begin{proof}
  Set $Y = Y_1\hipcap Y_2$; it suffices to note that the complement $X = S\setminus Y$ of $Y$
  is not a complete graph. It is seen that 
  \begin{ctext}
    $X = (X_1 \cap Y_2) \cup (X_2\cap Y_1)$.
  \end{ctext}
  If $x,y\in X_i$, $x \neq y$ then, by definition, there is $z\in Y_i$ such that
  $\{ x,y,x \}$ is a line of $\goth M$; we write $z = \{ x,y \}^{\infty_i}$.
  Take any distinct $x,y\in S$.
  \\
  If $x,y\in X_1,Y_2$ then $x,y$ are joinable; let $z = x\oplus y$ be the third element of
  $\LineOn(x,y)$. Then $z\in\LineOn(x,y)\subset Y_2$ and
    $z = \{ x,y \}^{\infty_{1}}\in Y_1$.
  \\
  Analogously, if $x,y\in X_2,Y_1$ then the third point $z = x\oplus y$ of $\LineOn(x,y)$
  lies on $Y_1\cap Y_2$.
  \\
  Take $x\in X_1,Y_2$, $y\in X_2,Y_1$ and suppose there is a line through $x,y$; again we take
  $z = x\oplus y$. By the above, 
    $z\notin X_1\cap Y_2,X_2\cap Y_1,Y_1\cap Y_2$.
  So, only the case $z\in X_1\cap X_2$ remains to be examined.
  \\
  By \ref{prop:cross-compl2},
    $X_1\cap X_2= \{c\}$ for a point $c$. 
  So, finally, we take 
    $x\in X_1 \cap Y_2$ collinear with $c$,
  and 
    $y\in X_2\cap Y_1$, $y\notin \LineOn(x,c)$ 
  and then $x,y \in X$ are not collinear in $\goth M$.
\end{proof}
%

\section{Examples: hyperplanes in $10_3$-configurations}\label{ssec:10-3}  


To give intuitions how the hyperplanes in well known configurations look like we enclose 
this Section.
The following can be proved by a direct inspection of all the $10_3$ configurations.
Some of these facts follow from more general theory developed in next sections,
but we give them right at the beginning to give intuitions
how the theory looks like.
It is known that there are exactly ten $10_3$-configurations
(see e.g.  \cite{obrazki}, \cite{betten}, \cite{klik:VC}, \cite{klik:binom}); 
names of the configurations in question are used mainly after \cite{lanman}).
\begin{prop} (schemes of Veldkamp spaces of the configurations enumerated below
  are presented in figures \ref{fig:kantor}-\ref{fig:unreal})
  \begin{sentences}\itemsep-2pt
  \item Desargues configuration has fifteen hyperplanes (cf. \cite{saniga}).
  \item The Kantor $10_3G$-configuration (Fig. \ref{fig:kantor}) and 
    the nightcap configuration (Fig. \ref{fig:unreal}) have seven 
    hyperplanes each.
  \item The fez configuration (Fig. \ref{fig:fez}) and 
    the headdress configuration (Fig. \ref{fig:head}) contain three 
    hyperplanes each.
  \item The basinet configuration (Fig. \ref{fig:basinet}) and 
    the overseashat  configuration (Fig. \ref{fig:overseas}) contain exactly 
    one hyperplane each.
  \item Every of the remaining three $10_3$ configurations does not contain any 
    hyperplane.
  \end{sentences}
\end{prop}
\begin{figure}[H]
\begin{center}
  \includegraphics[scale=0.3]{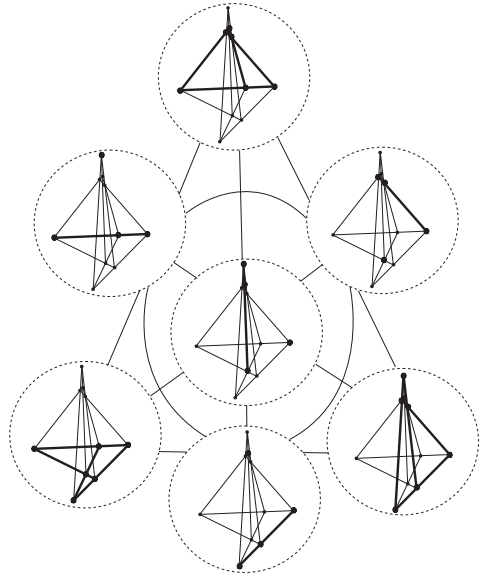}
\end{center}
\caption{The Veldkamp Space of the Kantor $10_3G$-Configuration $\VerSpace(3,3)$}
\label{fig:kantor}
\end{figure}

\begin{figure}[H]
\centering
\begin{minipage}{.45\textwidth}
  \centering
  \includegraphics[scale=0.3]{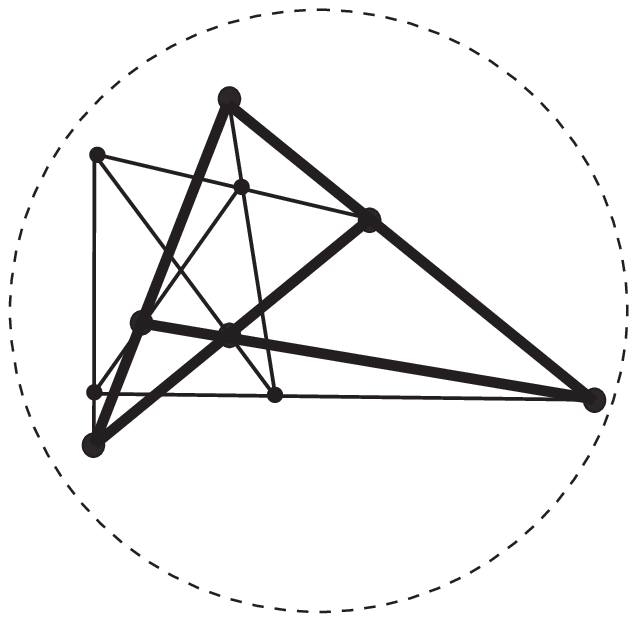}
  \captionof{figure}{The Veldkamp space  of the basinet configuration}
  \label{fig:basinet}
\end{minipage}%
\hfill
\begin{minipage}{.45\textwidth}
  \centering
  \includegraphics[scale=0.3]{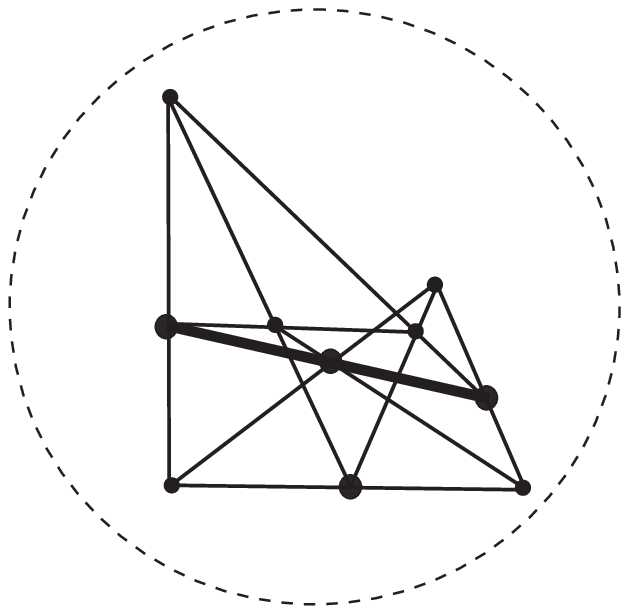}
  \captionof{figure}{The Veldkamp space of the overseas-cap configuration}
  \label{fig:overseas}
\end{minipage}
\end{figure}

\begin{figure}[H]
\begin{center}
  \includegraphics[scale=0.3]{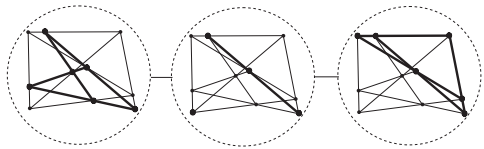}
\end{center}
\caption{The Veldkamp space of the fez configuration}
\label{fig:fez}
\end{figure}

\begin{figure}[H]
\begin{center}
  \includegraphics[scale=0.3]{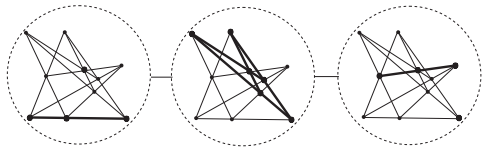}
\end{center}
\caption{The Veldkamp space of the headdress configuration}
\label{fig:head}
\end{figure}

As a consequence we can formulate
\begin{rem}
  There are binomial configurations (even quite small: $10_3$-confi\-gu\-ra\-tions)
  with exactly one hyperplane. This one may be the complement of
  a complete graph 
  (a $10_3$-configuration with exactly one Veblen subconfiguration)
  or not
  (a $10_3$-configuration whose unique hyperplane consists of a point and a line).
\end{rem}

\begin{figure}[H]
\begin{center}
  \includegraphics[scale=0.3]{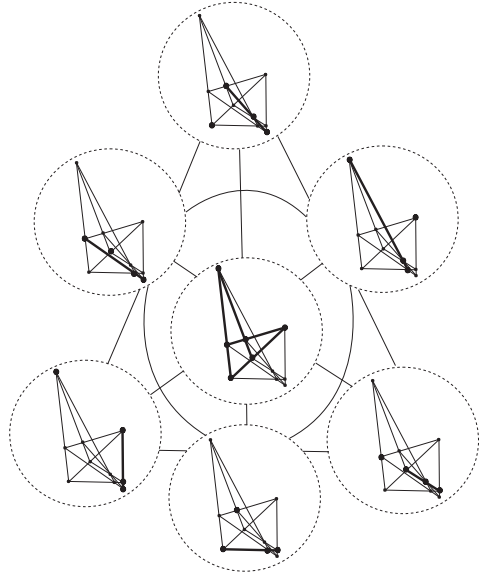}
\end{center}
\caption{The Veldkamp space of the nightcap configuration}
\label{fig:unreal}
\end{figure}


\section{A general approach: 
Hyperplanes of binomial \PSTS's with a maximal  complete graph inside}
\label{sec:withgraf}

Let us begin with the formal construction which makes more precise the statements
of \ref{pelny2horyzont}. 
Let $|X| = n$ and $0\notin X$, $W = X\cup\{  0 \}$.
Every \BSTS\ $\goth M$ with $\binom{n+1}{2}$ vertices and $K_X$ freely contained in it can be presented in the form
$K_X +^\mu {\goth V}$, defined below:

let 
  $\goth V$ be a $\binkonf(n,0)$-configuration 
and 
  $\mu$ be a bijection of $\sub_2(X)$ onto the  point set of $\goth V$. 
The point set of $K_X +^\mu {\goth V}$ is the union of the
set of vertices of $K_X$ and the point set of $\goth V$. 
The set of lines of $K_X +^\mu {\goth V}$
is the union of the set of lines of $\goth V$ and the family
  $\big\{ \{x,y,\mu(\{ x,y \})\}\colon x,y\in X, x\neq y \big\}$.
Up to an isomorphism, $K_X +^\mu {\goth V}$ can be in a natural way defined on the set $\sub_2(W)$
as its point set: 
we identify each $x\in X$ with the set $\{ 0,x \}$, 
and identify  each point $\mu(\{x,y\})$ of $\goth V$ with  the set $\{ x,y \}$, 
suitably transforming the line set of $\goth V$ and putting, formally, 
$\mu(\{ x,y \}) = \{ x,y \}$. 
Frequently, we write 
  ${(x,y)}^\infty = {\LineOn(x,y)}^\infty = \mu(\{x,y\})$ for distinct $x,y \in X$.

In the first step we shall characterize hyperplanes in a configuration ${\goth M} = K_X +^\mu {\goth V}$.
So, let $H$ be a hyperplane of $\goth M$. 
Let $V$ be the point set of $\goth V$; then $H_0 = H\cap V$ is a hyperplane of $\goth V$ or $H_0 = V$.
\def\whip{{\sim}}
\def\nowhip{\not\mathrel{\whip}}
We begin with several technical lemmas. For $x,y\in X$ we write
  $x \whip y$ when $x\neq y$ and ${(x,y)}^\infty\in H_0$ or $x = y \in X$. 
\begin{lem}\label{whip1}
  The relation $\whip$ is an equivalence relation.
\end{lem}
\begin{proof}
  It is evident that $\whip$ is symmetric and reflexive. So it remains to prove the transitivity of $\whip$.
  Let $x,y,z\in X$ be pairwise distinct. Assume that ${(x,y)}^\infty, {(x,z)}^\infty \in H_0$ and suppose that
  ${(y,z)}^\infty\notin H_0$. Then $H \cap \LineOn(y,z) \subset \{y,z\}$, as $H$ crosses {\em every} line of $\goth M$.
  Assume $y\in H$; from $\LineOn(y,{(x,y)}^\infty) \subset H$ we infer $x \in H$ and then $z\in H$ follows.
  Finally, $\LineOn(y,z)\subset H$, so ${(y,z)}^\infty\in H_0$.
\end{proof}
\begin{lem}\label{whip2}
  Let $x \in X$. If there is $z\in [x]_\whip \cap H_0$ then $[x]_\whip \subset H$.
\end{lem}
\begin{proof}
  From assumptions, $[x]_\whip = [z]_\whip$. Let $y \whip z$ be arbitrary. 
  Then ${(y,z)}^\infty \in H$ and $z \in H$ yield   $y \in H$.
\end{proof}
Write ${\cal X} = X\diagup\whip$. From \ref{whip2} we know that 
\begin{ctext}
  for every ${\goth a} \in {\cal X}$, either ${\goth a} \subset H$ or ${\goth a}\cap H = \emptyset$.
\end{ctext}
\begin{lem}\label{whip3}
  For every ${\goth a},{\goth b}\in{\cal X}$ if ${\goth a},{\goth b} \subset H$ then 
  ${\goth a} = {\goth b}$.
\end{lem}
\begin{proof}
  Let $x,y\in X$ such that ${\goth a} = [x]_\whip$ and ${\goth b} = [y]_\whip$. Let $x\neq y$. From assumptions,
  $x,y \in H$ and then $\LineOn(x,y)\subset H$ gives ${(x,y)}^\infty\in H$ i.e. $x\whip y$, as required.
\end{proof}
\begin{lem}\label{whip4}
  For every distinct ${\goth a},{\goth b}\in{\cal X}$, ${\goth a} \subset H$ or ${\goth b} \subset H$.
\end{lem}
\begin{proof}
  Let $x,y\in X$ such that ${\goth a} = [x]_\whip$ and ${\goth b} = [y]_\whip$. From assumptions, $x\nowhip y$
  i.e. ${(x,y)}^\infty\notin H$.
  But $\LineOn(x,y)$ crosses $H$, so $x\in H$ or $y \in H$. From \ref{whip2} we get the claim.
\end{proof}
Now, we are in a position to prove the first (main) characterization.
\begin{thm}\label{prop:hipinSSP:0}
  Let $H$ be a hyperplane of ${\goth M} = K_X +^\mu {\goth V}$ defined on the set $\sub_2(W)$, as
  introduced at the beginning of the section.
  Then there is a subset $A$ of $W$ such that $H = \hipa(A|W\setminus A)$.
\end{thm}
\begin{proof}
  From \ref{whip3} and \ref{whip4} we get that either $\whip$ has exactly two equivalence classes 
    ${\goth a}\subset H, {\goth b} \subset X\setminus H$ 
  or 
    ${\cal X} = \{ X  \}$.
  In the second case, $V \subset H$, and if there were $x\in H \cap X$ then $H$ is the point set of $\goth M$.
  So, 
     $H = V = \sub_2(X) = \hipa(\{ 0 \} | W \setminus \{ 0 \} )$.
  Let us pass to the first case. Note that $H$ is the union of three sets:
    $\mu(\sub_2({\goth a})) = \sub_2({\goth a})$,
    $\mu(\sub_2({\goth b})) = \sub_2({\goth b})$,
  and $\goth a$ which, under identification introduced before, corresponds to 
    $\big\{ \{ 0,x \}\colon x \in {\goth a} \big\}$.
  So, finally, $H$ can be written in the form 
    $\sub_2({\goth b}) \cup \sub_2({\goth a}\cup\{ 0 \})
    = \hipa({\goth b}|W\setminus{\goth b})$.
\end{proof}

Next, we are going to determine which ``bipartite'' sets $\hipa(A|X\setminus A)$ are hyperplanes of suitable
\BSTS's. To this aim one should know more precisely what is the number of complete graphs inside a given
configuration.

Let us recall the following construction
\par
Let $I = \{ 1,\ldots,m \}$ be arbitrary, let $n > m$ be an integer, and let
$X$ be a set with $n-m+1$ elements.
Let us fix an arbitrary $\binkonf(n-m,+1)$-configuration 
${\goth B} = \struct{Z,{\cal G}}$.
Assume that we have two maps $\mu,\xi$ defined:
  $\mu\colon I\longrightarrow Z^{\sub_2(X)}$ 
and
  $\xi\colon I\times I\longrightarrow S_X$, such that
$\xi_{i,i} = \id$, $\xi_{i,j} = \xi_{j,i}^{-1}$, and
$\mu_i$ is a bijection for all $i,j\in I$.
Set 
  $S = Z \cup (X\times I) \cup \sub_2(I)$ 
(to avoid silly errors we assume that the given three sets are pairwise disjoint).
On $S$ we define the following family $\lines$ of blocks
\begin{eqnarray}
  \label{SSP:lines1}
  \lines & = & {\cal G} 
  \\ \label{SSP:lines2}
  \strut & \cup & \text{ the lines of } \GrasSpace(I,2) 
  \\ \label{SSP:lines3}
  \strut & \cup & \left\{ \{ \{i,j\}, (x,i), (\xi_{i,j}(x),j)  \}
                  \colon \{i,j\}\in\sub_2(I), x \in X  \right\} 
  \\ \label{SSP:lines4}
  \strut & \cup & \left\{ \{ (a,i), (b,i), \mu_i(\{ a,b \}) \} 
                  \colon \{ a,b \}\in \sub_2(X), i\in I  \right\}.
\end{eqnarray}
Write 
  $m \bowtie^\mu_\xi {\goth B} = \struct{S,\lines}$. 
It needs only a straightforward (though quite tedious) verification to prove that 
{\em ${\goth M}:= m \bowtie^\mu_\xi {\goth B}$ is a $\binkonf(n,+1)$-configuration}. 

For each $i\in I$ we set 
  $Z_i = X\times\{ i \}$, $S_i = \{ e\in\sub_2(I)\colon i\in e \}$,
and 
  $X_i = Z_i \cup S_i$. 
Then 
{\em $\goth M$ freely contains $m$ $K_n$-graphs; these are $X_1,\ldots,X_m$}.
It is seen that the point $\{i,j\}$ is the ``perspective center" of two subgraphs $Z_i,Z_j$
of $\goth M$. So, we call the arising configuration 
{\em a system of perspectives of $(n-m+1)$-simplices}.
Define 
  $\mu_i\colon\sub_2(Z_i)\longrightarrow Z$ 
by the formula
  $\mu_i(\{ (x,i),(y,i) \}) = \mu(i)(\{ x,y \})$; 
the configuration $\goth B$ is the common `axis' of  the configurations 
  $\struct{Z_i,\sub_2(Z_i)} +_{\mu_i} {\goth B}$
contained in $\goth M$.
Let us denote $W = X\cup I$.
Without loss of generality we can assume that $Z = \sub_2(X)$
and each $(x,i)\in X\times I$ can be identified with the set $\{ x,i \}$.
After this identification $\sub_2(W)$ becomes the point set of $\goth M$.

The following is crucial:
\begin{thm}[{\cite{klik:binom}}]
  Let $\goth M$ be a $\binkonf(n,+1)$-configuration.
  $\goth M$ freely contains (at least) $m$ $K_{n}$-graphs
  iff 
    ${\goth M}\cong {m \bowtie^\mu_\xi {\goth B}}$
  for a $\binkonf(n-m,+1)$-configuration $\goth B$
  and a pair $(\mu,\xi)$ of suitable maps.
\end{thm}

Combining the results of \cite{klik:binom} and \cite{STP-4}
it is not too hard to prove the following criterion
\begin{prop}\label{prop:SSP:degen}
  Let ${\goth M} = {m \bowtie^\mu_\xi {\goth B}}$
  for a $\binkonf(n-m,+1)$-configuration $\goth B$
  and a pair $(\mu,\xi)$ of suitable maps.
  The following conditions are equivalent
  \begin{sentences}\itemsep-2pt
  \item
    $\goth M$ freely contains at least $m+1$ $K_n$ graphs.
  \item
    $\goth M$ contains at least $m+1$ $\binkonf(n,0)$-subconfigurations. 
  \item
    There is $x_0\in X$ such that $\xi(i,j)(x_0) = x_0$ for each pair $i,j\in I$.
  \end{sentences}
\end{prop}
From now on we assume that
\begin{ctext}
  $\goth B$ has $\sub_2(X)$ as its point set and
    ${\goth M} = {m \bowtie^\mu_\xi {\goth B}}$, 
  defined on the point set $\sub_2(W)$,
  freely contains exactly $m$ $K_n$-subgraphs;
\end{ctext}
this means that $\hipa(i) = \hipa(\{i\}|W\setminus\{i\})$ with $i\in I$ are the unique hyperplanes of 
$\goth M$ of the size $\binom{n}{2}$. 

From the above, \ref{prop:veldkam}, and \ref{rem:algdziel} we infer immediately
\begin{prop}\label{prop:hipinSSP:1}
  \begin{sentences}\itemsep-2pt
  \item\label{hipinSSP:sa} 
    Every set 
      $\hipa(J|W\setminus A)$ with $J\subset I$ is a hyperplane of $\goth M$.
      In particular, $\hipa(I|X)$ is a hyperplane of $\goth M$.
  \item\label{hipinSSP:brak}
    There is no $a\in X$ such that 
      $\hipa(\{ a \}|W\setminus \{\ a \})$ is a hyperplane of $\goth M$.
  \end{sentences}
\end{prop}


With the help of \eqref{wzor:hipcap} we get
\begin{equation*}
  \hipa(A\cup J \big| W\setminus(A\cup J)) = 
  \hipa(A|W\setminus A)\hipcap\hipa(J|W\setminus J)
\end{equation*}
for every $J\subset I$, $A\subset X$. So, from \ref{prop:hipinSSP:1} we get
\begin{ctext}
  if $J\subset I$, $A\subset X$ then \\ 
  $\hipa(A\cup J|W\setminus(A\cup J))$ is a hyperplane of $\goth M$ iff  
  $\hipa(A|W\setminus A)$ is a hyperplane of $\goth M$.
\end{ctext}
\begin{thm}\label{prop:hipinSSP:2}
  Let $A\subset X$. Then $\hipa(A|W\setminus A)$ is a hyperplane of $\goth M$ iff
  the following conditions are satisfied:
  \begin{sentences}\itemsep-2pt
  \item\label{dzhip1}
    $\sub_2(A)$ is a hyperplane of $\goth B$.
  \item\label{dzhip2}
    $A$ is invariant under every $\xi(i,j)$, $i,j\in I$.
  \item\label{dzhip3}
    $A$ is invariant under every $\mu_i$, $i\in I$, which means the following:
    if $\mu_i(x,y) = \{ u,v \}$ then $\{ x,y \}\subset A$ or $\{ x,y \}\subset X\setminus A$
    iff $\{ u,v \}\subset A$ or $\{ u,v \}\subset X\setminus A$.
  \end{sentences}
\end{thm}
\begin{note}
  In some applications there is no 
  way to present, in a natural way, the underlying $\binkonf(n-m,+1)$-configuration $\goth B$ 
  as a structure defined on the family of 2-subsets of a $(n-m+1)$-element set;
  {\em natural}   from the point of view of the geometry of $\goth B$.
  Then one can take one of $\mu_i$'s as basic and replace $\goth B$ as its coimage under $\mu_i$
  defined on $\sub_2(X)$.
  Under this stipulation the condition \eqref{dzhip3} of \ref{prop:hipinSSP:2}
  is read as follows:
  \begin{enumerate}[(i')]\setcounter{enumi}{2}
  \item
    if $\mu_i(x,y) = \mu_j(u,v)$ for some $i,j\in I$ then $\{ x,y \}\subset A$ or $\{ x,y \}\subset X\setminus A$
    iff $\{ u,v \}\subset A$ or $\{ u,v \}\subset X\setminus A$.
  \end{enumerate}
\end{note}
\begin{proof}
  We use notation of the definition of a system of perspectives of simplices presented
  in this Section.
  The symbol $a \oplus b$ means the third point $\LineOn(a,b)\setminus\{ a,b \}$ on the line
  through $a,b$ (if the line exists).
  Note that 
    $\sub_2(W\setminus A) = \sub_2(I)\cup \sub_2(X\setminus A) \cup (X\setminus A)\boxtimes I$, 
  where
    $(X\setminus A) \boxtimes I := \big\{ \{ b,i \}\colon i\in I,\,b\in X\setminus A\big\}$. 
  Therefore,
  \begin{ctext}
    $H = \hipa(A \big| W\setminus A) = 
      \sub_2(A) \cup \sub_2(I) \cup \sub_2(X\setminus A) \cup  (X\setminus A)\boxtimes I$, 
  \end{ctext}
  where $A\subset X$.
  In view of \ref{prop:hipinSSP:1}\eqref{hipinSSP:sa} 
  without loss of generality we can assume that	$A\neq X$.
 \par
  Since $\sub_2(I)\subset H$,
  \begin{multline}\label{cz1}
    \text{if a line } L \text{ of } {\goth M} \text{  has two points common with } \sub_2(I) 
    \text{  then } L\subset H, \\    \text{and each line of }  {\goth M}  \text{ of the form \eqref{SSP:lines2} crosses }  H.
  \end{multline}
 \par\strut\quad
  Assume that $H$ is a hyperplane of $\goth M$.
  Since $\sub_2(X)$ is the point set of $\goth B$, and $\goth B$ is a subspace of $\goth M$ 
  right from definition, 
    $\hipa(A|X\setminus A) = \sub_2(A)\cup\sub_2(X\setminus A)$ 
  is a hyperplane of $\goth B$ or $\sub_2(X)=\sub_2(A)$.
  The latter  means $A=X$, which contradicts assumptions. So, \eqref{dzhip1} follows.
 \par
  On the other hand, converting the above reasoning we easily prove that \eqref{dzhip1} implies
  \begin{multline}\label{cz2} 
    \text{ if a line } L \text{ of } {\goth M} \text{  has two points common with } \sub_2(A) \cup \sub_2(X\setminus A) 
    \text{  then } L\subset H,  \\ \text{ and each line of } {\goth M} \text{ of the form 
    \eqref{SSP:lines1} crosses } H.
  \end{multline}
  \par\strut\quad
  Next, let us pass to the lines of $\goth M$ of the form \eqref{SSP:lines3}.
  Suppose that 
    $a\notin A$, $i_1,i_2\in I$, $i_1\neq i_2$. 
  Then
    $\{a,i_1\}, \{ i_1,i_2 \} \in H$, so
    $\{ a,i_1 \}\oplus \{ i_1,i_2 \} = \{ \xi(i_1,i_2)(a),i_2 \} \in H$.
  Consequently, $\xi(i_1,i_2)(a)\notin A$.
  This justifies condition \eqref{dzhip2}.
  \par
  Considering all the points expressible in the form 
    $\{ a,i_1 \}, \{ i_1,i_2 \}, \{ a',i_2 \}\in H$, i.e. with $a,a'\notin A$, $i_1,i_2\in I$, $i_1\neq i_2$  
  and afterwards considering their `product' 
    $\{ a,i_1 \} \oplus \{ i_1,i_2 \}$ and
    $\{ a,i_1 \} \oplus \{ a',i_2 \}$ 
  we see that, conversely, \eqref{dzhip3} implies  
  \begin{multline}\label{cz3} 
    \text{ if a line } L \text{ of } {\goth M} \text{  has two points common with } H,
    \text{ two in } (X\setminus A)\boxtimes I \\ \text{ or one in } (X\setminus A)\boxtimes I
    \text{ and the second in } \sub_2(I)
    \text{  then } L\subset H,  
    \\ \text{ and each line of } {\goth M} \text{ of the form 
    \eqref{SSP:lines3} crosses } H.
  \end{multline}
  \par\strut\quad
  Finally, we pass to the lines of $\goth M$ of the form \eqref{SSP:lines4}.
  Let $p = \{ a_1,a_2 \}\in \sub_2(X)$,   
  and $b\in X$, $i\in I$, $q = \{i,b\}$.
  Suppose $p,q$ are collinear in $\goth M$; this means 
  $\mu_i(b,b') = \{a_1,a_2\}$ for a point $b'\in X$
  and $r:= p \oplus q = \{ i,b' \}$. Set $d = \{ b,b' \}$.
  Assume that $p\in H$ 
  i.e. 
  $p\subset A$ or $p\subset X\setminus A$.
  Then $q \in H$ iff $r\in H$ i.e.
   $q \in (X\setminus A)\boxtimes I$
  iff 
   $r \in (X\setminus A)\boxtimes I$.
  Finally: $b\in A$ iff $b'\in A$,
  so  $d \subset A$ or $d \subset (X\setminus A)$ follows.
  \par
  Next, let $q\in H$. Then $p\in H$ iff $r\in H$ yields
  that the implication ($b,b'\in A$ or $b,b'\notin A$ $\implies$ $p \subset A$
  or $p\subset X\setminus A$) holds.
  So, finally, we have proved \eqref{dzhip3}.
 \par
  Converting the reasonings above we see that \eqref{dzhip3} implies
  \begin{multline}\label{cz4}
    \text{ if a line } L \text{ of } {\goth M} \text{  has two points common with } H,
    \text{ two in } (X\setminus A)\boxtimes I \\ \text{ or one in } (X\setminus A)\boxtimes I
    \text{ and the second in } \sub_2(A)\cup\sub_2(X\setminus A)
    \text{  then } L\subset H,  
    \\ \text{ and each line of } {\goth M} \text{ of the form 
    \eqref{SSP:lines4} crosses } H.
  \end{multline}
  Gathering together the conditions \eqref{cz1}, \eqref{cz2}, \eqref{cz3}, and \eqref{cz4} 
  we obtain that the conjunction 
  \eqref{dzhip1} \& \eqref{dzhip2} \& \eqref{dzhip3}
  implies that $H$ is a hyperplane of $\goth M$.
\end{proof}
There do exist \PSTS's which satisfy the assumptions \eqref{dzhip1}-\eqref{dzhip3};
as examples known in the literature we can quote quasi Grassmannians, comp. Subsect. 
\ref{ssec:qgras} and, in particular, \ref{thm:hipinqgras}.
Another class of examples is shown in \ref{exm:dzhip}.
\begin{exm}\label{exm:dzhip}
  Let $I = \{ 1,\ldots,m \}$, $X = \{ a,a,b,b' \}$ and
  ${\goth B} = \GrasSpace(X,2)$ be the Veblen configuration.
  Set 
    $\xi(i,j) = \xi(j,i)(a,a',b,b')=(a',a,b',b)$,
    $\mu_i(x,y) = \{x,y\}$
  for all $i,j\in I$, $x,y\in X$.
  Let us put
    ${\goth M} := m \bowtie^\mu_\xi {\goth B}$.
  Then $\goth M$ is a system of perspectives of $m$ tetrahedrons.
  It freely contains exactly $m$ graphs $K_{m+3}$, so it contains
  exactly $m$ hyperplanes of the form $\hipa(\{x\}|X\setminus\{x\}) = \hipa(x)$
  with $x\in I\cup X$. However, it contains the hyperplane
  $\hipa(\{ a,a' \} | \{ b,b' \} \cup I)$ which is not $\hipcap$-generated
  from the $\hipa(x)$'s.
\end{exm}

We close this section with a characterization of geometries on hyperplanes
  $\hipa(A\cup J|B\cup E)$ of $\goth M$, where
  $\{ A,B \}$ is a decomposition of $X$ and $\{ J,E \}$ is a decomposition of $I$.
So, let us assume that \eqref{dzhip1}-\eqref{dzhip3} of \ref{prop:hipinSSP:2} hold.
\begin{prop}
  Let $k = |J|$, $m = |I|$.
  \begin{sentences}\itemsep-2pt
  \item
   $\hipa(J|W\setminus J)$ is the union of the generalized Desargues configuration
   $\GrasSpace(J,2)$ and the system 
     $k \bowtie^{\mu\restriction J}_{\xi\restriction J\times J}{\goth B}$
   of perspectives of $k$ simplices $K_X$.
  \item
   $\hipa(A | W\setminus A)$ is the union of the restriction 
     ${\goth B}\restriction  \sub_2(A) =: {\goth B}'$
   and the system 
     $m \bowtie^{\mu'}_{\xi'} {{\goth B}'}$,
   where $\xi'(i,j) = \xi(i,j)\restriction A$ and
     $\mu'(i) = \mu(i)\restriction \sub_2(A)$
   for all $i,j\in I$,
   of perspectives of $m$ simplices $K_A$.
  \end{sentences}
\end{prop}


\section{Examples: structure of hyperplanes in \BSTS's of some known classes}
\label{sec:examples}


\subsection{Hyperplanes in generalized Desargues configurations}\label{ssec:gras}

Recall: a $\binkonf(n,0)$-configuration $\goth M$ freely contains $n$ graphs $K_{n-1}$
(the maximal possible amount) iff $\goth M$ is isomorphic to
the {\em generalized Desargues configuration} 
  $\GrasSpace(X,2) = \struct{\sub_2(X),\{ \sub_2(Z)\colon Z\in\sub_3(X) \}}$ 
for a $X$ with $|X| = n$ (cf. \cite{klik:binom}, \cite{STP-4}).
The class of generalized Desargues configurations appears in many applications,
even in physics: \cite{doliwa1}, \cite{doliwa2}.

\begin{thm}\label{thm:hipingras}
  Let  $H\subset \sub_2(X)$. Write ${\goth H} = \GrasSpace(X,2)$. 
  The following conditions are equivalent
  \begin{sentences}\itemsep-2pt
  \item\label{wwar1}
    $H$ is a hyperplane of $\goth H$.
  \item\label{hypa:forma:2}
    There is a proper non void subset $Z$ of $X$ such that 
    $H = \hipa(Z|X\setminus Z)$.   
  \end{sentences}
  Consequently, $\VSpace({\GrasSpace(n,2)}) = \projgeo(n-2,2)$ (comp. \cite{saniga}).
\end{thm}
\begin{proof}
  Let $H$ be as required in \eqref{hypa:forma:2} and let 
  $L =  \sub_2(A)$ for a $A\in\sub_3(X)$ be a line of $\goth H$. 
  If $A\subset Z$ or $A \subset X\setminus Z$ then 
  $L\subset H$. Assume that $A\not\subset Z,X\setminus Z$. Then
  there are $i,j\in A$, $i\in\notin Z$, $j\notin (X\setminus Z)$
  So: $i \in X\setminus Z$, $j\in Z$. Write $A=\{i,j,l\}$.
  If $l\in Z$ then $\{ j,l \}\in L\cap H$, if $l\in X\setminus Z$
  then $\{ i,l \}\in L\cap H$.
  Finally, we note that if $L\cap\sub_2(Z)\neq\emptyset$ then there is no point in
  $L\cap \sub_2(X\setminus Z)$. And similarly conversely.
  This proves that $H$ is a subspace of $\goth G$, so, finally, \eqref{wwar1} is valid.
  \par
  The implication \eqref{wwar1}$\implies$\eqref{hypa:forma:2} is immediate after \ref{prop:hipinSSP:0}.
  \par
  Finally, there are $\frac{2^n - 2}{2} = 2^{n-1} -1$ suitable decompositions of $X$; this determines
  $\dim({\VSpace({\GrasSpace(X,2)})})$.
\end{proof}
Recall after \cite{perspect} that each set 
\begin{equation}\label{def:gwiazda}
  S(i) = \{ a\in\sub_2(X)\colon i\in a \}
\end{equation}
is a complete graph freely contained in $\GrasSpace(X,2)$.
In consequence of \ref{rem:algdziel} and \ref{thm:hipingras},  
the hyperplanes 
  $\hipa(i) = \hipa(\{i\}|X\setminus\{i\})$ with $i\in X$
(`binomial hyperplanes' of $\GrasSpace(X,2)$) generate via the operation $\hipcap$
all the hyperplanes of $\GrasSpace(X,2)$.

Note that each of the two components 
  ${\cal A } = \sub_2(A)$ and  ${\cal A}' = \sub_2(A')$ with $A'=X\setminus A$ of
$\hipa(A|A')$ is a binomial configuration. These two components are {\em complementary} (unconnected) in
the following sense:
\begin{ctext}
  if $a\in {\cal A}$ and $a'\in{\cal A}'$ then 
  $a,a'$ are uncollinear in $\GrasSpace(X,2)$;
  ${\cal A}'$ consists of the points that are uncollinear with every point
  in ${\cal A}$, and conversely.
\end{ctext}


\subsection{Hyperplanes of quasi Grassmannians}\label{ssec:qgras}

First, we recall after \cite{skewgras} a construction of quasi Grassmannians.

Let us fix two sets:  
$Y$ such that $|Y| = 2(k-1)$ for an integer $k$ and 
$X_0$ such that $X_0\cap Y = \emptyset$,
  $X_0 = \{1,2\}$ or $X_0 = \{ 0,1,2 \}$.
We put $X = X_0 \cup Y$;
Then $n:= |X| = 2k$ or $n=2k+1$, resp.
The points of the {\em quasi Grassmannian $\vergras_n$} are the elements of $\sub_2(X)$. 
The lines of ${\goth R}_n$ are of two sorts:
the lines of  $\GrasSpace(X,2)$ which miss $\{ 1,2 \} =: p$ remain unchanged.
The class of lines of $\GrasSpace(X,2)$ through $p$
(i.e. the sets $\sub_2(Z)$ with $1,2\in Z\in\sub_3(X)$) 
is removed; instead, we add the following sets
  $\big\{ \{1,2\}, \{ 1,2j+2 \}, \{ 2,2j+1 \} \big\}$,
  $\big\{ \{1,2\}, \{ 1,2j+1 \}, \{ 2,2j+2 \} \big\}$
(we adopt a numbering of $Y$ so that $Y = \{ 3,4,5,6,\ldots,2k \}$).
It is seen that $\vergras_n$ is a $\binkonf(n,0)$-configuration.

\begin{prop}[{\cite{skewgras}}]
  The maximal complete $K_{n-1}$-graphs contained in $\vergras_n$ are 
  exactly all the sets  $S(i)$ (as defined by \eqref{def:gwiazda}) with $i\in X_0$.
\end{prop}

Write
  $D(i):= \sub_2(X)\setminus S(i) = H(\{ i \}| X\setminus \{ i \})$.
Then the following is immediate
\begin{fact}\label{fct:bininqgras}
  When $n=2k+1$, then $D(0) \cong \vergras_{2k}$.
  For every $n$, $D(1)$ and $D(2)$ yield in $\vergras_n$ (binomial) subconfigurations
  isomorphic to $\GrasSpace(n-1,2)$.
\end{fact}

\begin{thm}\label{thm:hipinqgras}
  Set $T = \{ 2,3,\ldots,k \}$, $q_t = \{ 2t,2t-1 \}$ for $t \in T$.
  Then $Y = \bigcup_{t\in T} q_t$.
  The family 
  $\hipy(\vergras_n)$ consists of the sets
  \begin{equation}\label{eq:hipinqgras}
    \hipa\left(A \cup \textstyle{\bigcup_{t \in J} q_t} \Big| 
     (X_0\setminus A) \cup \textstyle{\bigcup_{s\in T\setminus J} q_s} \right)
  \end{equation}
  with arbitrary $A \subseteq X_0$, $J\subseteq T$ such that
  $(A,J) \neq (\emptyset,\emptyset),(X_0,T)$.
  Consequently,
    $\VSpace(\vergras_{2k}) = \projgeo(k-1,2)$ and
    $\VSpace(\vergras_{2k+1}) = \projgeo(k,2)$.
\end{thm}
\begin{proof}
  An elementary computation shows that each set of the form \eqref{eq:hipinqgras}
  is a hyperplane of $\vergras_n$.  
  Let us represent $\vergras_n$ as a suitable system of perspectives.
  Define, first, for $i\in X_0$ and distinct $x,y\in Y$: 
    $\mu_i(\{ x,y \}) = \{ x,y \}$.
  Next, we set 
    $\xi(1,2)(2j+1,2j+2) = (2j+2,2j+1)$ for every $j=1,...,k-1$
  and 
    $\xi(1,0) = \xi(2,0)(x) = x$ for $x\in Y$, if $0\in X_0$. 
  We have obtained two maps 
    $\mu\colon X_0\longrightarrow {\sub_2(Y)}^{\sub_2(Y)}$
  and 
    $\xi \colon X_0\times X_0 \longrightarrow S_{Y}$.
  It is seen that 
    $\vergras_n \cong m \bowtie_{\xi}^{\mu}\GrasSpace(X_0,2)$.
  A $\xi$-invariant subset of $Y$ is the union of several sets of the form 
  $q_t$, $t \in T$. It is seen that such a union is $\mu$-invariant.
  From \ref{prop:hipinSSP:2} we infer that each hyperplane of $\vergras_n$
  has form \eqref{eq:hipinqgras}.

  To complete the proof it suffices to note that there are $2^{|X_0|}$ decompositions
  of $X_0$, $2^{k-1}$ decompositions of $T$, and 
  $\frac{2^{|X_|+k-1}}{2} - 1 = 2^{|X_0|+k-2} -1$ decompositions of $W$ which yield a 
  hyperplane. Substituting $|X_0|=2$ and $|X_0|=3$ we get 
  $|\hipy(\vergras_{2k})| = 2^k - 1$ and
  $|\hipy(\vergras_{2k+1})| = 2^{k+1} - 1$,
  which closes the proof.
\end{proof}


\subsection{Hyperplanes of multi-veblen configurations}\label{ssec:mveb}

Recall another fact: a $\binkonf(n,0)$-configuration freely contains $n-2$ graphs $K_{n-1}$
if it is a {\em (simple) multi-veblen configuration}. 
The multi-veblen configurations can be also defined by means of a direct construction.
Let us recall, briefly, after \cite{pascvebl} this construction.

Let $X$ be an $n$-set disjoint with a two-element set $p$ and $\cal P$ be a graph defined on $X$.
The points of the configuration
${\bf M}(X,p,{\cal P})$ are the following:
$p$, $s_i$ with $s\in p$, $i\in X$, and $c_{i,j}$ with
$\{ i,j\}\in\sub_2(X)$.
The lines are: the sets of the form
$\{ p,a_i,b_i \}$, the sets 
  $\{ a_i,a_j,c_{i,j} \}$, $\{ b_i,b_j,c_{i,j} \}$ for $\{i,j\}\in{\cal P}$, $p = \{ a,b \}$,	
and
  $\{ a_i,b_j,c_{i,j} \}$, $\{ b_i,a_j,c_{i,j} \}$ for $\{i,j\}\notin{\cal P}$, $p = \{ a,b \}$,
finally: 
  the sets $c_u,c_v,c_w$, 
where 
  $\{ u,v,w \}$ is a line of $\GrasSpace(X,2)$.
It is seen that after the identification 
  $s_i \leftrightarrow \{ s,i \}$ for $s\in p,\;i\in X$
and
  $c_{i,j}\leftrightarrow \{ i,j\}$ for $i,j\in X$
the structure
${\bf M}(X,p,{\cal P})$ can be defined on the set
$\sub_2(X\cup p)$ and then the configuration  
is easily seen to be a $\binkonf(n,2)$-configuration.
One can observe that each of the sets
$$
  S(i)  = \{ s_i,c_{i,j}\colon j\in X\setminus\{ i\},\;s\in p \}
    \leftrightarrow \{ q\in\sub_2(X\cup p)\colon i\in q \}
$$
is a complete $K_{n+1}$-graph freely contained in ${\bf M}(X,p,{\cal P})$.
Moreover, the complement 
$$
  H(i) = \{p,s_j,c_{j,l}\colon s\in p,\;j,l\in X\setminus\{i\}
    \leftrightarrow \{ q\in\sub_2(X\cup p)\colon i\notin q \}
$$
of $S(i)$ is a multi-veblen configuration
  ${\bf M}(X\setminus\{ i \},p,{\cal P}\restriction (X\setminus\{ i \}))$.
So, we can write simply 
  $H(i) = \hipa(\{i\}|X\setminus\{ i \})$.
It is known that a mutliveblen $\binkonf(n,2)$-configuration is either a generalized
Desargues configuration or it has exactly $n$ maximal freely contained complete graphs.
\begin{prop}\label{prop:hipinmveb}
  Assume that ${\bf M}(X,p,{\cal P}) =: {\goth M}$ is not a generalized Desargues
  configuration. Then every binomial hyperplane of $\goth M$ has the form
    $\hipa(\{ i \}|X\cup\{p\}\setminus\{i\})$ with $i\in X$.
  Each hyperplane of $\goth M$ has form
    $\hipa(A|(X\setminus A)\cup p)$
  for $\emptyset\neq A\subset X$.
\end{prop}
\begin{proof}
  It suffices to present $\goth M$ in the form $n\bowtie_\xi^\mu \GrasSpace(p,2)$.
  Indeed, we observe, first, that $\GrasSpace(p,2)$ is a trivial structure with a single point and no line.
  Next, we put $\mu_i(a_i,b_i) = p$ for all $i\in I$, $\{a,b\} = p$.
  Finally, 
    $\xi(i,j)(a,b) = (a,b)$ when $\{ i,j \}\in{\cal P}$ and $\xi(i,j)(a,b) = (b,a)$ otherwise.
  Comparing definitions we see that 
    ${\bf M}(X,p,{\cal P})\cong n\bowtie_\xi^\mu \GrasSpace(p,2)$. 
  \par
  To complete the proof we make use of \ref{prop:hipinSSP:0} and \ref{prop:hipinSSP:2}: 
  a hyperplane of $\goth M$ has form
    $\hipa(J|(I\setminus J)\cup p) \hipcap \hipa(A|(p\setminus A)\cup I)$ 
  for a subset $J$ of $I$ and
  an ($\mu,\xi$)-invariant subset $A$ of $p$. 
  From \ref{prop:SSP:degen} we get that a non void proper subset of $p$ is a one-element set, and
  such a subset of $p$ is invariant
  only when $\goth M$ is a generalized Desargues configuration, which closes our proof.
\end{proof}
Let us apply \ref{prop:hipinmveb} to the particular case 
  ${\cal P} = N_X$ (the empty graph defined on $X$);
it is known that ${\bf M}((X,p,{\cal P})$ is 
the structure $\VerSpacex(3,|X|)$ dual to
the combinatorial Veronesian $\VerSpace(X,3)$ (see \cite{combver} and Section \ref{ssec:ver}).
So, $\VerSpacex(n,3)$  has all its binomial hyperplanes of the same
geometrical type: the dual Veronesian $\VerSpacex(n-1,3)$.


\subsection{Hyperplanes of combinatorial Veronesians}\label{ssec:ver}

Next, let us pay attention to the class of combinatorial Veronese spaces
defined in \cite{combver}.  
  Write $X: = \{ a,b,c \}$ for pairwise distinct $a,b,c$.
  Generally, if $f = a^i b^j c^m$ is a multiset with the elements in $X$ we put $|f| = i+j+m$.

The {\em combinatorial Veronese space}
  $\VerSpace(3,k) = \VerSpace({\{ a,b,c \}},k)$
  is the configuration
  whose points are the multisets $a^i b^j c^m$, $i+j+m = k$:
  the elements of $\msub_k(\{ a,b,c \})$, and
  whose lines have form $e  X^i$, $i + |e| = k$.
It is a 
  $\binkonfo(k,2)$-configuration. 
  \par
  It is known that
    $\VerSpace(X,2) \cong \GrasSpace(4,2)$ 
  so the hyperplanes of $\VerSpace(X,2)$ are, generally, known.
  \par
  Let ${\goth M} = \VerSpace(X,k)$.
  It is known (cf. \cite{veradjac}, \cite{klik:binom}) that
  $K_{k+1}$-graphs freely contained in $\goth M$ have form $\msub_k(A)$,
  where $A\in\sub_2(X)$, the complement of such a graph is the set 
  $z \msub_{k-1}(X)$, where $\{z\} = X\setminus A$, 
  so it yields a (binomial) subspace of $\VerSpace(X,k)$
  isomorphic to $\VerSpace(X,k-1)$.
\begin{rem}
  Note that the set
  $H = \{ a^2c, b^2a, c^2b \}$ yields a hyperplane in every  Veblen subconfiguration  
  contained in $\VerSpace(3,3)$, but 
  $H$ {\em  is not } a hyperplane of $\VerSpace(3,3)$:  
  a 3-element {\em \bf anticlique} of a $10_3$-configuration 'suffices'
  for at most $3\times 3 = 9$  lines only (i.e. at most $9$ lines intersect such a 3-set).
\end{rem}

Let us generate via $\hipcap$ the hyperplanes, starting form the
binomial hyperplanes of a $\VerSpace(3,k)$.  
\begin{itemize}\itemsep-2pt
\item
  There are three hyperplanes 
  $H_1(u) = u\msub_{k-1}(X)$ and three their complements 
    $\overline{H_1(u)}  = \msub_k(\{ x,y \})$,
  $X = \{ u,x,y \}$.
\item
  Let us compute: 
    $H_1(x)  \hipcap H_1(y) = 
     \{u^k\} \cup x y \msub_{k-2}(X) =: H_2(u)$,
  where $X = \{ x,y,u \}$.
  \par
  The complement of $H_2(u)$ has the form 
    $\overline{H_2(u)} = x\msub_{k-1}(\{u,x\}) \cup y\msub_{k-1}(\{u,y\})$.
\item
  Let us  compute: 
    $H_2(a) \hipcap H_2(b) = H_2(c)$.
\item
  The properties of $\hipcap$ yield 
  $H_1(x)\hipcap H_2(y) = H_1(u)$ for $x\neq y$
  and $X = \{ x,y,u \}$.
\item
  Let us compute again:
    $H_1(x) \hipcap H_2(x) = X^k \cup abc\msub_{k-3}(X)$.
\end{itemize}
We have got seven hyperplanes of $\VerSpace(X,k)$. 
\begin{thm}\label{thm:hipinver}
  The above are  {\em\bf all   the} hyperplanes of $\VerSpace(X,k)$.
  So, $\VSpace({\VerSpace(3,k)}) = \projgeo(2,2)$.
\end{thm}
\begin{proof}\def\porz{\prec} \def\triop{\oplus}
  In the first step we present ${\goth M} := \VerSpace(X,k)$ as a system of perspectives of simplices. 
  Recall that   $\VerSpace(X,k)$ is a $\binkonf(k,2)$-configuration.
  In what follows we shall keep a fixed cyclic order $\porz$, say $(a \porz b \porz c \porz a)$ of the elements of $X$.
  Note that $\GrasSpace(X,2) \cong \VerSpace(X,1)$ is a single 3-element line.
  Set ${\goth B} = abc\VerSpace(X,k-3)$, it is a $\binkonf(k,-1)$-subconfiguration of $\VerSpace(X,k)$.
  Moreover, it is the intersection of three complements of the three maximal complete subgraphs $\msub_k(\{x,y\})$,
  $\{x,y\}\in\sub_2(X)$ of $\goth M$.
  As usually, we write $\triop$ for the (partial) binary operation `the third point on the line through'.
  Frequently, writing $x,yz,z$ below we mean any $x,y,z$ such that $X = \{x,y,z\}$.
  
  Next, let $Z = \{ 1,\ldots,k-1  \}$, then $|Z| = k-1$.
  For every $z \in X$ we define
  \begin{equation}
    \nu_z \colon Z \ni s \longmapsto x^s y^{k-s}, \text{  where } x \porz y,\; \{ z,y,z \} = X.
  \end{equation}
  So, $\goth M$ contains three copies: 
    $F_z = \msub_{k}(\{ x,y \})\setminus (\msub_k(\{ x,z \}) \cup \msub_k(\{ y,z \})) = \nu_z(Z)$ 
  ($X = \{ x,y,z \}$) 
  of $K_Z$.
  
  Next, for $z \in Z$ and distinct $i,j\in Z$ we define 
  \begin{equation}
    \mu_z(\{i,j\}) = \mu_z(i,j) = \nu_z(i) \triop \nu_z(j).
  \end{equation}
  It is easy to compute that 
    $\mu_z(i,j) = x^i y^i z^{k-2i}\in abc \VerSpace(X,k-3)$, 
  so we have defined  a surjection 
    $\mu_z \colon \sub_2(Z) \longrightarrow {\goth B}$.

  Finally, for distinct $x,y \in X$ and $s \in Z$ we define the map 
    $\xi_{x,y}\colon Z \longrightarrow Z$
  by the formula
  \begin{equation}
    \xi_{x,y}(s) = k-s
  \end{equation}
  and we set $\xi_{x,x} = \id$.
  The following holds for  $\{ x,y, z \} = X$ and $i,j \in Z$:
  \begin{equation*}
    \nu_x(i), \nu_y(j) \text{ collinear in } {\goth M} \iff j = \xi_{x,y}(i); \quad
    \nu_x(i) \triop \nu_y(k-i) = z^k.
  \end{equation*}
  So, in fact, for each $\{x,y,z\} = X$ we have a perspective 
    $\xi_{x,y} \colon F_x \longrightarrow F_y$ with the centre $z^k$
  determined by the formula 
    $\xi_{x,y}(\nu_x(i)) = \nu_y(\xi_{x,y}(i))$.
  Then 
    $\VerSpace(X,k) \cong 3 \bowtie_{\xi}^{\mu} {\goth B}$.

  Suppose that $\goth M$ contains a hyperplane $H$ of the form $\hipa(A, X \cup(Z\setminus A))$
  with $A \subset Z$.
  In view of \ref{prop:hipinSSP:2}, $A$ is a $(\mu,\xi)$-invariant subset of $Z$.
  Without loss of generality we can assume that $1 \in A$ and then $\{ 1,k-1 \} \subset A$.
%
%
  We get $\mu_z(1,k-1) = x y z^{k-2} = \mu_y(k-1,2)$; then $2\in A$, because $A$ is $\mu$-invariant
  (here, we make use of \ref{prop:hipinSSP:2}(\ref{dzhip3}'), in fact). 
  Consequently,
  $k-2 \in A$ as well.

  Step by step, we end up with $\{i,k-i\}\subset A$ for every $i\in Z$, so $A = Z$, which, 
  by \ref{prop:hipinSSP:0} and \ref{prop:hipinSSP:2} proves the theorem.
\end{proof}
%


\section{Ideas, hypotheses, and so on $\ldots$}

\subsection{Veldkamp space labeled}

As we see, the number of free subgraphs of a \BSTS\ $\goth M$ does not determine $\goth M$.
Also, the number of of its hyperplanes and the types of geometry on hyperplanes
do not determine $\goth M$.
Clearly, $\VSpace({\goth M})$ says only about $|\hipy({\goth M})|$.

Recall that if $\goth M$ is a $\binkonf(n,0)$-configuration with a free $K_{n-1}$-subgraph
then each hyperplane of $\goth M$ is either a $\binkonf(n,-1)$ or the union
of two unconnected $\binkonf(k_1,0)$ and $\binkonf(k_2,0)$-subconfigurations
of $\goth M$ with $k_1 + k_2 = n$, $k_1,k_2\geq 2$..
Suppose that for every $k < n$ we have the list 
  ${\cal M}_k$ of  $\binkonf(k,0)$-configurations.
Let ${\bf T}({\goth M})$ be $\VSpace({\goth M})$ with its points labelled by the 
types of respective hyperplanes, i.e. by symbols 
from ${\cal M}_{k-1}$ or unordered pairs of symbols from ${\cal M}_k\times{\cal M}_{n-k}$.
It seems that ${\bf T}({\goth M})$ may uniquely characterize $\goth M$.

\subsection{Problem}

In all the examples which were examined in the paper a hyperplane of a \BSTS\
(if exists) is either connected, and then it is a binomial maximal subspace, or it is 
the union of two unconnected (in a sense: mutually complementary) binomial subspaces.
Is this characterization valid for {\em arbitrary} \BSTS.


\bigskip

\par\noindent\small
Authors' address:\\
Krzysztof Petelczyc, 
Ma{\l}gorzata Pra{\.z}mowska, Krzysztof Pra{\.z}mowski\\
Institute of Mathematics, University of Bia{\l}ystok\\
K. Cio{\l}kowskiego 1M\\
15-245 Bia{\l}ystok, Poland\\
e-mail: \\
{\ttfamily kryzpet@math.uwb.edu.pl}, 
{\ttfamily malgpraz@math.uwb.edu.pl},
{\ttfamily krzypraz@math.uwb.edu.pl}

\end{document}